\documentclass[final]{amsart}

\usepackage{amsmath}
\usepackage{mathrsfs}
\usepackage{amssymb}
\usepackage{amsfonts}
\usepackage{stmaryrd}
\usepackage{enumerate}
\usepackage[usenames]{color}
\usepackage[pdftex]{graphicx}
\usepackage{epstopdf}
\usepackage[colorlinks=true, pdfstartview=FitV, linkcolor=blue, citecolor=blue, urlcolor=blue,pagebackref=false]{hyperref}

\usepackage[foot]{amsaddr}

\usepackage{lipsum}

\newcommand{\HOX}[1]{\marginpar{\footnotesize #1}}
\newtheorem{theorem}{Theorem}[section]
\newtheorem{lemma}[theorem]{Lemma}

\newtheorem{remark}[theorem]{Remark}

\newcommand{\eps}{\varepsilon}
\renewcommand{\phi}{\varphi}

\renewcommand{\tilde}{\widetilde}
\newcommand{\R}{\mathbb R}

\newcommand{\C}{\mathbb C}
\renewcommand{\S}{\mathbb{S}}

\newcommand{\tl}{\tilde}

\newcommand{\curl}{\nabla\times}
\newcommand{\grad}{\nabla}
\newcommand{\cout}[1]{}

\newcommand{\e}{\varepsilon}
\def \bfo {\begin {eqnarray*} }
\def \efo {\end {eqnarray*} }
\def \ba {\begin {eqnarray*} }
\def \ea {\end {eqnarray*} }
\def \beq {\begin {eqnarray}}
\def \eeq {\end {eqnarray}}

\def \p {\partial}

\title[3-dimensional invisibility cloaking]{The blow-up of electromagnetic fields in 3-dimensional invisibility cloaking for Maxwell's equations}

\begin{document}

\author{Matti Lassas}
\address[Matti Lassas]{Department of Mathematics and Statistics, University of Helsinki}

\author{Ting Zhou}
\address[Ting Zhou]{Department of Mathematics, Northeastern University.}

\maketitle

\noindent {\bf Abstract:} {\it
 Transformation optics constructions have allowed the design of cloaking devices that steer electromagnetic, acoustic and quantum waves around a region without penetrating it, so that this region is hidden from external observations. The {proposed} material parameters 
 are anisotropic, and singular at the interface between the cloaked region and the cloaking device.
The presence of these singularities causes various mathematical problems and physical effects on the interface surface. In this paper, we analyze the 3-dimensional cloaking for Maxwell's equations
 when there are sources or sinks present inside the cloaked region. In particular, we consider nonsingular approximate invisibility cloaks based on the truncation of the singular transformations. 
 We analyze the limit
 of solutions when the approximate cloaking approaches the ideal cloaking
 in the sense of distributions.
  We show that the solutions in the approximate cloaks converge to a distribution that contains Dirac's delta
distribution supported {on the interface surface}. 
In particular, this implies that 
the limit of solutions are not measurable functions, 
{making them} outside of those classes of functions that have earlier
been used in  the models of the ideal invisibility cloaks. Also, we give a rigorous meaning for the ``extraordinary surface voltage effect"
considered in physical literature of invisibility cloaks. 
}
\medskip

\section{Introduction}

The transformation optics-based cloaking design, since proposed in 2006 in \cite{Leo,PenSchSmi}, has attracted most attention among many theoretical proposals for realizing invisibility, with widely reported experiments and data accumulated due to the development of metamaterials. The basic idea is that, the coordinate transformation invariance for certain systems, such as those describing electromagnetic/optic or acoustic wave propagations, makes modifying the background medium of a local region $\Omega$ in a certain way undetectable when observe far away.  
In particular, when {\em singular} spatial transformations are used, like the one that blows up a point into a small bounded domain $D$, this undetectable customization of medium (known as the push-forward by the singular transformation) can have the structure of a fixed layer $\Omega\backslash\overline D$ (the cloaking device) surrounding ``arbitrary" media in $D$ (the cloaked region). Namely, whatever in $D$ along with the cloaking device is invisible! Such singular coordinate transformations to create a ``hidden pocket'' surrounded by
a layer of degenerate material,
were proposed in 2003 in
 \cite{GLU} for electrostatics as a counter-example of uniqueness of an inverse problem, known as the Calder\'on's problem. In \cite{Leo, PenSchSmi} a customized cloaking layer was proposed to be created of metamaterials.
Such layer should be designed so that it bends the light rays or detour the electromagnetic waves away from the cloaked region and return them back, as if they penetrate the region straightly, making observation outside indistinguishable from that of the empty background space.

The main difficulty in analyzing this scheme {rigorously} lies on the prescribed {\em singular} (non-regular) medium in the cloaking layer,
{see discussion in \cite{ALP,GKLU3,GKLU5}}. To be more precise, we consider in this paper the electromagnetic cloaking, where the prescribed medium in the cloaking layer $\Omega\backslash\overline D$ has electric permittivity $\tl\varepsilon(x)$, magnetic permeability $\tl\mu(x)$. Since they are  the push-forward of the ambient medium parameters by a singular transformation (detailed in Section \ref{sec:sing}), {one of the eigenvalues
of $\tl\varepsilon(x)$ and $\tl\mu(x)$}
 degenerates to zero at the exterior of the cloaking interface $\partial D^+$. Suppose $D$ is filled with {an arbitrary medium $(\mu_1(x), \varepsilon_1(x))$} to be cloaked. 
Let us consider the waves propagating in  
domain $\Omega\subset \R^3$ consisting of medium whose
permittivity and permeability are given by
$(\tl\mu,\tl\varepsilon)$ in $\Omega\backslash\overline D$ (the cloaking layer\slash device)  and $(\mu_1,\varepsilon_1)$ in $\overline D$ (the cloaked region). 
{Such cloaking gives counterexamples for uniqueness of the inverse problems 
for Maxwell's equations.
As we want to consider also cloaking of active bodies, we add a current source $\tl J$ in the domain $D$.
To study this, we will consider the case when $\mu_1(x)=\mu_0$ and $\varepsilon_1(x)=\varepsilon_0$
are constants. We note that the results of this paper can with small modifications
be extended to the case where  $\mu_1(x)-\mu_0$  and $\varepsilon_1(x)-\varepsilon_0$ are compactly
supported functions in $D$.

Since $(\tl\mu,\tl\varepsilon)$ are singular near $\partial D$, the key mathematical question
is, in what sense the solutions to Maxwell's equations 
\[\textrm{curl}\, \tl E-i\omega \tl B=0,\qquad \textrm{curl}\, \tl H+i\omega \tl D=\tl J,\qquad \textrm{ on }\;\Omega\backslash\overline D \cup D,\]
exist. Here the constitutive relations are given by
\[\tl D=\left\{\begin{array}{ll}\tl\varepsilon\tl E & \mbox{ on }\;\Omega\backslash\overline D,\\ \varepsilon_0\tl E &\mbox{ on }\;D.\end{array}\right. \qquad \tl B=\left\{\begin{array}{ll}\tl\mu\tl H & \mbox{ on }\;\Omega\backslash\overline D,\\ \mu_0\tl H &\mbox{ on }\;D.\end{array}\right.\] 
Moreover, one would like to specify what kind of boundary conditions should appear at the interface $\partial D$, where the singularity happens. There have been many different proposals to
model the above Maxwell's equations with singular coefficients and different
boundary conditions have been proposed, depending on what spaces the solutions of the equations
are assumed to belong. In the present paper our aim is to consider approximative
cloaks and consider what is the limit of the solutions as it approaches
an ideal cloaking. Before formulating our results, let us review the earlier proposed models.\\




In \cite{GKLU}, 
the concept of {\em finite energy solution} (FES) was formulated. These
are distributional solutions to Maxwell's equations for which all physical
fields, {that is, the electric and magnetic fields $\tl E$ and $\tl H$ 
 as well as
the electromagnetic fluxes  $\tl D$ and $\tl B$  
are $\C^3$-valued measurable functions on $\Omega$. The ``tilde'' here refers to the fact
that these solutions satisfy Maxwell's equation with 
the cloaking material parameters $\tl\varepsilon$ and $\tl\mu.$}
Moreover, the energy norm with degenerate weight is required to be finite, that is 
\[\int_{\Omega\backslash\overline D} \big(\tl\varepsilon^{jk}\tl E_j\overline{\tl E_k}+\tl\mu^{jk}\tl H_j\overline{\tl H_k}\big)~dV+\int_{D}\big(\varepsilon_0^{jk}\tl E_j\overline{\tl E_k}+\mu_0^{jk}\tl H_j\overline{\tl H_k}\big)~dV<\infty,\]
where, as everywhere below, we have used the Einstein summating convention by summing over indecies $j$ and $k$ that appear both as sub- and superindecies.
 It is shown in \cite{GKLU} that a hidden boundary condition must be satisfied by the FES at the interior of the interface, namely
\[\nu\times \tl E|_{\partial D^-}=0,\qquad \nu\times\tl H|_{\partial D^-}=0.\]
These, physically known as PEC and PMC conditions, form an over-determined set of boundary conditions for Maxwell's equations in $D$, and could be only possibly satisfied for cloaking passive media {(i.e., $\tl J=0$ or $\tl J$ is a non-radiating source that would produce in the free space $\R^3$ a compactly supported electromagnetic field)}. But for generic $\tl J$, including ones arbitrarily close to $0$, there is no FES-solution.

In \cite{We1, We2}, solutions with various boundary conditions at $\partial D^\pm$ are considered in the context of self-adjoint
extensions of  
the Maxwell's {operator}, that are compatible with energy conservation. There, the fields $\tl E$ and $\tl H$
are assumed to be measurable functions that are in the
domain of the  closed quadratic form, as well as in the domain of the
corresponding self-adjoint operator. We call these the operator theoretic (OT) solutions.
{It is shown there that the Maxwell operator $A$ with degenerate coefficients
$\tl\varepsilon(x)$ and $\tl\mu(x)$ is essentially self-adjoint and its unique self-adjoint extension  has a domain
$\mathcal D(A)\subset L^2(\Omega)$ such that
$(E,H)\in \mathcal D(A)$ have one-sided traces (from the outside of $D$) satisfying}
\[\nu\times\tl E|_{\partial D^+}=0,\qquad \nu\times\tl H|_{\partial D^+}=0,\]
due to the degeneracy of parameters at $\partial D^+$. With regular medium to be cloaked in $D$, functions
{$(E,H)\in\mathcal D(A)$ have  also one-side traces, from the inside of $D$, satisfying}
\begin{equation}\label{eqn:Weder-2}\nu\cdot(\textrm{curl }\tl H)|_{\partial D^-}=\nu\cdot(\textrm{curl }\tl E)|_{\partial D^-}=0.\end{equation}
Notice that this is from the point of view that self-adjoint extensions in the domain $\Omega$ are direct sums of those in $\Omega\backslash\overline D$ and $D$, hence the solutions are considered completely decoupled from each other.\\

In \cite{GKLU_2,GKLU-jst,GKLU-pnas,KOVW, KSVW, LZ,Ngu,Ngu2,NgVo,NgVo2}, for the purpose of both physical realization and analysis of the waves behavior, the regularization of the singular scheme is studied as the fundamentals of approximate cloaking design. In this paper, we will consider the same 
{approximation}
scheme as in \cite{LZ} for electromagnetic waves. The detailed {non-singular approximation} will be introduced in Section \ref{sec:sing}. Roughly speaking, instead of using the transformation that blows up a point into the {cloaked region $D=B_1$ where $B_R\subset \R^3$  is the  ball
of radius $R$ centred at the origin,}
 in this idealized spherical geometric setting, we use a non-singular transformation that blows up a small ball $B_\rho$ with radius $\rho$ (the regularization parameter) into $B_1$. The transformation medium is then regular for fixed $\rho>0$ and appears as a small inhomogeneity in the empty space. The well-understood solutions for regular media provide a tool to scrutinize the extreme case, that is taking the limit $\rho\rightarrow0$.

{ As seen above, there are different, partly contradicting, alternatives for rigorous models of  ideal invisibility cloaks. As physical attempts to 
build up an invisibility cloak from metamaterial are always based on approximate constructions, we propose here a change of
the point of view for defining a model of an ideal invisibility cloak:
We consider a limit of the solutions in an approximate cloak as $\rho\rightarrow0$,  in the sense ofdistributions (i.e., generalized functions),
and describe the plausible solutions in an ideal cloak as
{these limits of the solutions}. 
We show that the solutions in the approximative cloaking structures} converge to a distribution that contains Dirac's delta
distribution supported on the interface $\partial D$. 
This implies that the set of measurable functions might be too small as the domain of the singular Maxwell's operator,
when considered as the limit of electromagnetic approximate cloaking. {We note that the ideal and 
the approximate cloaks we study here have no energy absorption, that is, the conductivity is zero.}


{
The results of the paper can be summarised as follows: 
Let $\tilde E_\rho$  and $\tilde H_\rho$  denote the electric and the magnetic fields in an approximative cloak in $\Omega$ as described in \eqref{eq: approx. cloak},
satisfying the 
boundary condition $\nu\times \tilde E_\rho|_{\p B_2}=f$. 
We consider the limit, in the sense of distributions,
$\tilde E_0=\displaystyle\lim_{\rho\to 0} \tilde E_\rho$ and $\tilde H_0=\displaystyle\lim_{\rho\to 0} \tilde H_\rho$ of the solutions 
of the Maxwell equations, that is, the limit when
the approximative cloak approaches a perfect cloak. Then

\begin{enumerate}

\item We show that the limit $(\tilde E_0,\tilde H_0)$ is not in the class of finite energy solutions.
This is compatible with the non-existence result of  finite energy solutions given in \cite{GKLU}.

\item We show that the limit $(\tilde E_0,\tilde H_0)$ is not a measureble function and thus it 
does not belong in the class of operator theoretic solutions studied in  \cite{We1}.
However, if $(\tilde E_0,\tilde H_0)$ is decomposed to a sum of a measurable function $(E_m,H_m)$
and a delta-distribution supported on the surface $\p D$, the  function  part $(E_m,H_m)$
satisfies the non-standard  boundary conditions  \eqref{eqn:Weder-2} 
on  $\p D$ described in \cite{We1},
that is, certain traces of $(E_m,H_m)$ are zero.
Surprisingly, we observe that the corresponding traces of the fields $\tilde E_\rho$ and $\tilde H_\rho$ on $\p D$ do not converge
to zero,
see Remark \ref{rmk:tang_cond}.

\item The convergence of the solutions in the so-called virtual space and physical space are different: In
the physical space the limit of $(\tilde E_\rho,\tilde H_\rho)$ contains a delta-distribution part but in the virtual space, images of the fields  given in (\ref{eq: virtual space images}),
converge in the $L^2$-sense as described in \cite{LZ}.

\item The blow up of the wave, that is, the appearance of the delta-distribution in the limit  $\rho\to 0$
happens with all frequencies. Thus the observed blow up is different to the destruction of the cloaking 
effect appearing at the eigenfrequencies of the inside of the cloak 
studied in  \cite{GKLU_2,GKLU-jst,KOVW}.

\item The high concentration of waves on $\p D$
has been observed in the physical literature by Zhang et. al. \cite{ZCWK}, and is called the ``extraordinary surface voltage effect".
We give a rigorous formulation, as a convergence result in the weak topology of $H^{-1}(\Omega)$, of this effect.
\end{enumerate}

We note that in this paper we  consider only the most commonly used approximative cloaks that are equivalent to multiplying
 the singular  electric permittivity $\tl\varepsilon(x)$ and magnetic permeability $\tl\mu(x)$ by a characteristic function.
Such approximation could be implemented also in different ways, see  \cite{Ammari1,
 Bao1,Bao2,GKLU_2, KOVW, KSVW,LaZ,LST,LZ,
 Ngu,Ngu2,NgVo,NgVo2}, and an interesting question is how the delta-distribution part of the limit solution $(E_0,H_0)$
depends on {how the approximation is used}. 
However, this question is outside the scope of this paper. 

}

The rest of this paper is organized as following. In Section \ref{sec:sing}, we describe in details the transformation optics-based ideal cloaking and the regularized approximate cloaking in the spherical geometric setting. The main result is formulated in Theorem \ref{thm:main} and the proof is presented in Section \ref{sec:proof}, where the Dirac's delta distribution is proved rigorously to appear at the exterior of the interface $\partial D^+$, as a limit of the high concentrated electromagnetic waves in the approximate cloak. In Section \ref{sec:motif}, an example is presented to demonstrate the physical motivation and insight of the existence of such condition. Our estimates rely heavily on the spherical harmonics and asymptotic of Bessel and Hankel functions, of which we provide some basics in the Appendix.

\section{Singular ideal electromagnetic cloaking and the regularization}\label{sec:sing}

How the transformation invariance of the system amounts to the undetectability of certain non-ambient media with respect to the observation is best seen when formulated as inverse boundary value problems for underlying PDEs (for near-field measurements) or inverse scattering problems (for far-field measurements). This is also why the layer-structured medium for cloaking was first studied in \cite{GLU} for electrostatics as a counter-example of uniqueness of an inverse problem, known as the Calder\'on's problem. Here we consider a similar construction for electromagnetism, that is for Maxwell's equations. In particular, we prescribe time-harmonic incident waves with time frequency $\omega$, which results in the time harmonic equations as following.

Let $\Omega$ denote a bounded domain in $\R^3$ with smooth boundary. 
Consider Maxwell's equations for time-harmonic electric and magnetic fields $(E, H)$, viewed as 1-forms,
\begin{equation}\label{eq:Maxwell-eqn}
\nabla\times E=i\omega B,\quad \nabla\times
H=-i\omega D+J\qquad\mbox{in\ \
$\Omega$}
\end{equation} 
where the 2-forms $B$, $D$ and $J$ denote the magnetic induction, the electric displacement and the current density respectively, Furthermore, we specify the constitutive relations as
\begin{equation}\label{eqn:DB}
D=\varepsilon E,\qquad B=\mu H,
\end{equation}
where $\varepsilon$ and $\mu$ represent the permittivity and permeability of the material in $\Omega$ . 

In the following we assume that $\varepsilon,\mu$ are in
$L^\infty(\Omega)^{3\times 3}$, and satisfy 
\begin{equation}\label{eq:regular-medium}
c_m|\xi|^2\leq \xi^T\varepsilon(x)\xi\leq c_M|\xi|^2,\ \ \
c_m|\xi|^2\leq \xi^T\mu(x)\xi\leq c_M|\xi|^2
\end{equation}
for some constants $c_m,c_M>0$ and all
$x\in\Omega$ and $\xi\in\R^3\backslash\{0\}$,
We remark that \eqref{eq:regular-medium} are physical conditions for
{\it regular} EM media. Given $J\in L^2(\Omega)^3$, we
denote  by $\mathcal{C}_{\varepsilon,\mu, J}^\omega$ a subset of $H^{-\frac 1 2}(\textrm{Div};\partial\Omega)\times H^{-\frac 1 2}(\textrm{Div};\partial\Omega)$ given by
\[\mathcal C_{\varepsilon,\mu,J}^\omega:=\;\Big\{(\nu\times E|_{\partial\Omega}, \nu\times H|_{\partial\Omega})~|~\textrm{$E, H\in H(\textrm{curl};\Omega)$ satisfy \eqref{eq:Maxwell-eqn} and \eqref{eqn:DB} {with $J$}}\Big\}\]
where 
$$H(\mbox{curl};\Omega):=\{\mathbf{u}\in
L^2(\Omega)^3;\textrm{curl }\mathbf{u}\in L^2(\Omega)^3\},$$ and 
\[H^{-\frac 1 2}(\mbox{Div};\partial\Omega):=\{\mathbf{f}\in H^{-\frac
1 2}(\partial\Omega)^3;\mathbf{f}\cdot\nu=0\ \mbox{a.e. on
$\partial\Omega$}\ \mbox{and}\ \mbox{Div}\, \mathbf{f}\in H^{-\frac
1 2}(\partial\Omega)\}\]
with $\mbox{Div}$ denoting the surface divergence on
$\partial\Omega$.
Also, we denote $\mathcal C_{\varepsilon,\mu}^\omega:=\mathcal C_{\varepsilon,\mu,J}^\omega$
with $J=0$.
%

 The set $\mathcal C_{\varepsilon,\mu,J}^\omega$ is known as the {\it
Cauchy data set} which encodes the full exterior (boundary) measurements of electromagnetic fields. The inverse problem is then to understand the dependence of $\mathcal C_{\varepsilon,\mu,J}^\omega$ on the parameters $(\varepsilon,\mu)$, in order to recover the latter from the former. Knowing this, the design of invisibility is to find the parameters to be prescribed in a cloaking device such that its Cauchy data is indistinguishable from that of the vacuum background, independent of the object to be cloaked in the device.  \\

First, let us define $M^{-T}=(M^{-1})^T$ for an matrix $M$. 
Consider a transformation
$x=F(y):\Omega\rightarrow\tilde\Omega$ between two bounded domains $\Omega, \tl\Omega\subset\R^3$ with smooth boundaries. Assume $F$ is bi-Lipschitz and orientation-preserving, and denote
by $M:=DF(y)=(\frac{\partial{x}_i}{\partial y_j})_{i,j=1}^3$ the
Jacobian matrix of $F$. The pull-back fields of the solution $(E,H)\in
H(\mathrm{curl};\Omega)\times H(\mathrm{curl};\Omega)$ to (\ref{eq:Maxwell-eqn}) by $F^{-1}$, are given in $\tilde\Omega$ by 
\begin{equation}\label{eq:pullback-EH}\begin{split}
\tilde E(x)=&(F^{-1})^*E(x):=(M^{-T}E)\circ F^{-1}(x),\\
\tilde H(x)=&(F^{-1})^*H(x):=(M^{-T}H)\circ F^{-1}(x),
\end{split}\end{equation}
\begin{equation}\label{eq:pullback-J}
\tilde J(x)=(F^{-1})^*J(x)=([\det(M)]^{-1}M J)\circ F^{-1}(x).
\end{equation}
Then we have that  $(\tilde E, \tilde H)\in
H(\mathrm{curl};\tilde\Omega)\times
H(\mathrm{curl};\tilde\Omega)$ satisfies Maxwell's equations
\begin{equation}\label{eq:trans Maxwell}
\tilde\nabla\times\tilde E=i\omega\tilde\mu\tilde H,\ \
\tilde\nabla\times\tilde
H=-i\omega\tilde\varepsilon\tilde
E+\tilde J\qquad \mbox{in\ \ $\tilde\Omega$},
\end{equation}
where $\tilde\nabla\times$ denotes the curl in the $x$-coordinates, and $\tilde\varepsilon, \tilde\mu$ are
the push-forwards of $\varepsilon,\mu$ via $F$, defined by
\begin{equation}\label{eq:pushfw-eps}
\tilde\varepsilon(x)=F_*\varepsilon(x):=([\det(M)]^{-1}{M\cdot\varepsilon\cdot
M^T})\circ F^{-1}(x),
\end{equation}
and similarly for $\tilde\mu=F_*\mu$.
%
Moreover, if we assume 
$F|_{\partial\Omega}=\mathrm{Id}$, using Green's
identity, it is directly verified that
\[\mathcal C^\omega_{\varepsilon,\mu,J}=\mathcal C^\omega_{F_*\varepsilon,F_*\mu,(F^{-1})^*J}.\]

This summarizes the basics of transformation optics in a rather general
setting, which we shall make essential use of in the construction of 
singular ideal cloaking and regularized approximate cloaking. Throughout the paper, we will focus on a spherical geometry of design. 

{ Next, we consider singular coordinate transformations.}
Consider the map
\begin{equation}
\label{eqn:F1-sing}
F_1:\, B_2\backslash\{0\}\,\rightarrow\,B_2\backslash\overline{B_1},\quad
F_1(y)=\left(1+\frac{1}{2}|y|\right)\frac{y}{|y|},\;\;\;0<|y|<2
\end{equation}
which blows up the origin to $B_1$ while keeping the boundary $\partial B_2$ fixed. In the cloaking layer $x\in B_2\backslash\overline{B_1}$, we prescribe the EM material parameters given by
\begin{equation}\label{eq:cloaking-medium}
\tilde\mu(x)=\tilde\eps(x)=({F_1}_*I)(x):=\left.\displaystyle\frac{(DF_1)\cdot I\cdot (DF_1)^T}{|\det(DF_1)|}(y)\right|_{y=F_1^{-1}(x)}
\end{equation}
where $I$ is the identity matrix, representing the homogeneous vacuum background material. 
In the region $B_1$ to be cloaked, we consider an arbitrary
but {\it regular} EM medium $(\eps_0, \mu_0)$ satisfying \eqref{eq:regular-medium}, i.e., $\tilde\mu(x)=\mu_0(x)$ and $\tilde\eps(x)=\eps_0(x)$ for $x\in B_1$, which can be viewed as the push-forwards of $(\mu_0, \eps_0)$ in $B_1$ by $F_2=\mathrm{Id}$. We denote the ``glued"  transformation by 
\begin{equation}\label{eq:trans-F}
F=(F_1, F_2):~(B_2\backslash\{0\}, B_1)\rightarrow (B_2\backslash\overline{B_1}, B_1).
\end{equation}
Noticing that $F_1$ is a radial dilation and by some simple calculations, we have 
that
\begin{equation}\label{eq:spherical-cloaking-medium}
\tilde\mu(x)=\tilde\varepsilon(x)=2\frac{(|x|-1)^2}{|x|^2}\mathbf{e}_r+2\mathbf{e}_\theta,\quad
1<|x|<2,
\end{equation}
where $\mathbf{e}_r$ and $\mathbf{e}_\theta$ are respectively, the
unit projections along radial and angular directions, i.e.,
$\mathbf{e}_r=I-\hat{x}\hat{x}^T,\
\mathbf{e}_\theta=\hat{x}\hat{x}^T,\ \hat{x}={x}/{|x|}.$ It is
readily seen that as one approaches the cloaking interface $\partial B_1^+$
the medium in the cloaking device becomes singular, in the sense that $\tilde\varepsilon$ and
$\tilde\mu$ no longer satisfy the condition \eqref{eq:regular-medium} (the eigenvalue along the radial direction degenerates). 


\subsection{Construction of regularized approximate cloaking}

For approximate acoustic cloaking by regularization, Kohn et al., in
\cite{KOVW}, proposed blowing up a small ball $B_\rho$ to $B_1$
using a nonsingular transformation $F_\rho$ which degenerates to the
singular transformation $F_1$ in \eqref{eqn:F1-sing} as
$\rho\rightarrow0$, while Greenleaf et al., in \cite{GKLU_2},
proposed truncating the singular medium in \eqref{eq:spherical-cloaking-medium} to $B_2\backslash\overline{B_R}$ for $R>1$. For the present study, we shall focus on the `blow-up-$B_\rho$-to-$B_1$' regularization.

Let $0<\rho<1$ denote a regularizing parameter and set
\begin{equation}\label{eq:a-b}
a=\frac{2(1-\rho)}{2-\rho},\;\;\;\;b=\frac{1}{2-\rho}.
\end{equation}
Consider the nonsingular transformation from $B_2$ to $B_2$ defined by
\begin{equation}\label{eq:trans-F_rho}
x:=F_\rho(y)=\left\{\begin{array}{ll}F_\rho^{(1)}(y)=(a+b|y|)\frac{y}{|y|},
&\hbox{ for } \rho<|y|<2,\\
F_\rho^{(2)}(y)=\frac{y}{\rho},& \hbox{ for } \|y|\leq\rho.\end{array}\right.
\end{equation}
Our approximate cloaking device is obtained by the push-forward of a
homogeneous medium in $B_2\backslash \overline{B_\rho}$ by
$F_\rho^{(1)}$. Suppose we hide a regular but arbitrary uniform EM
medium $(\eps_0,\mu_0)$ in the cloaked region $B_1$. Then the
corresponding EM material parameter in $B_2$ in the {\it physical space} is
\beq\label{eq: approx. cloak}
(\tl\eps_\rho(x),\tl\mu_\rho(x))=\left\{\begin{array}{cl}((F_\rho^{(1)})_*I, (F_\rho^{(1)})_*I),\;\;&\hbox{ for } 1<|x|<2,\\ (\eps_0,\mu_0),&
\hbox{ for } |x|<1.\end{array}\right.
\eeq
The EM fields $(\tl E_\rho, \tl
H_\rho)\in H(\mbox{curl};B_2)\times H(\mbox{curl};B_2)$
corresponding to $\{B_2;\tl\varepsilon_\rho,\tilde\mu_\rho\}$ then
satisfy Maxwell's equations
\begin{equation}\label{eq:MW-phy-1}
\left\{\begin{array}{l}
\nabla\times\tl E_\rho=i\omega\tl \mu_\rho\tl H_\rho,\;\;\;\;\nabla\times\tl H_\rho=-i\omega\tl\eps_\rho\tl E_\rho+\tl J\;\;\;\;\mbox{in }\;B_2,\\
\nu\times\tl E_\rho|_{\partial B_2}=f\in
H^{-1/2}(\mbox{Div};\partial B_2)
\end{array}\right.
\end{equation}
where $\tilde J\in (L^2(B_1))^3$ is the current source. At this point, we assume that it is supported in a smaller ball $B_{r_1}$ with radius $0<r_1<1$. 
Then the pull-back EM fields
\beq\label{eq: virtual space images}
(E_\rho, H_\rho)=
((F_\rho)^*\tl E_\rho, (F_\rho)^*\tl H_\rho)\in
H(\mbox{curl};B_2)\times H(\mbox{curl};B_2)
\eeq
satisfy Maxwell's equations in the {\em virtual space} of $y$
with parameters 
\[(\eps_\rho(y),\mu_\rho(y))=\left\{\begin{array}{cl}(I, I)\;\;&\rho<|y|<2,\\((F_\rho^{(2)})^*\eps_0,(F_\rho^{(2)})^*\mu_0)\;\;&|y|<\rho \end{array}\right.\]
and 
\[J=(F_\rho^{(2)})^*\tilde J.\]
Moreover, the observation $\mathcal C^\omega_{\tilde\varepsilon_\rho,\tilde\mu_\rho}$ of the whole cloaking object $\{B_2;\tilde\varepsilon_\rho,\tilde\mu_\rho\}$ is shown to be identical to $\mathcal C^\omega_{\varepsilon_\rho,\mu_\rho}$, that of a small inhomogeneity of radius $\rho$ in the background vacuum space. In particular, the regularized cloaking is expected to converge to the ideal cloaking as $\rho$ shrinks. The order of such convergence was discussed in \cite{KOVW,LZ} for both Helmholtz equations and time-harmonic Maxwell's equations. Here we reproduce some calculations from \cite{LZ} for Maxwell's equations in order to pursue our discussion.

{ Due to the two-layered structure,  the following notations for EM fields will be adopted
\[
\tl{E}_\rho:=
\left\{\begin{array}{ll} \tl{E}_\rho^{+}(x), &\hbox{for }x\in B_2\backslash \overline{B_1},\\
\tl{E}_\rho^{-}(x),&  \hbox{for }x\in B_1,\end{array}\right.\quad
\tl{H}_\rho:=
\left\{\begin{array}{ll} \tl{H}_\rho^{+}(x), &\hbox{for }x\in B_2\backslash \overline{B_1},\\
\tl{H}_\rho^{-}(x),&  \hbox{for }x\in B_1,\end{array}\right.
%
%
%
\] 
in the {\em physical space} and
\[
E_\rho:=
\left\{\begin{array}{ll} {E}_\rho^{+}(y), &\hbox{for }y\in B_2\backslash\overline{B_\rho},\\
{E}_\rho^{-}(y),&  \hbox{for }y\in B_\rho,\end{array}\right.\quad
{H}_\rho:=
\left\{\begin{array}{ll} {H}_\rho^{+}(y), &\hbox{for }y\in B_2\backslash\overline{B_\rho},\\
{H}_\rho^{-}(y),&  \hbox{for }y\in B_\rho,\end{array}\right.
%
\] in the {\em virtual space},}
and the fields satisfy the following transmission problems
\begin{equation}\label{eq:MW-phy-2}
\left\{\begin{array}{l}\nabla\times \tl E^+_\rho=i\omega\tl\mu_\rho(x)\tl H^+_\rho,\;\;\;\;\nabla\times\tl H^+_\rho=-i\omega\tl \eps_\rho(x)\tl E^+_\rho\;\;\;\;\mbox{in }\; B_2\backslash\overline{B_1},\\
\nabla\times\tl E^-_\rho=i\omega\mu_0\tl H^-_\rho,\;\;\;\;\nabla\times\tl H^-_\rho=-i\omega\eps_0\tl E^-_\rho+\tilde J\;\;\;\;\mbox{in }\;B_1,\\
\nu\times\tl E^+_\rho|_{\partial B_1^+}=\nu\times\tl E^-_\rho|_{\partial B_1^-},\;\;\;\;\nu\times\tl H^+_\rho|_{\partial B_1^+}=\nu\times\tl H^-_\rho|_{\partial B_1^-},\\
\nu\times\tl E^+_\rho|_{\partial B_2}=f.\end{array}\right.
\end{equation}
and
\begin{equation}\label{eq:MW-vir-2}
\left\{\begin{array}{l}\nabla\times E^+_\rho=i\omega H^+_\rho,\;\;\;\;\nabla\times H^+_\rho=-i\omega E^+_\rho\;\;\;\;\mbox{in }\; B_2\backslash\overline{B_\rho},\\
\nabla\times E^-_\rho=i\omega\mu_\rho(y)H^-_\rho,\;\;\;\;\nabla\times H^-_\rho=-i\omega\eps_\rho(y)E^-_\rho+J\;\;\;\;\mbox{in }\;B_\rho,\\
\nu\times E^+_\rho|_{\partial B_\rho^+}=\nu\times E^-_\rho|_{\partial B_\rho^-},\;\;\;\;\nu\times H^+_\rho|_{\partial B_\rho^+}=\nu\times H^-_\rho|_{\partial B_\rho^-},\\
\nu\times E^+_\rho|_{\partial B_2}=f.\end{array}\right.
\end{equation}

\smallskip

Now we are ready to present our main theorem. 

\begin{theorem} \label{thm:main}
Let $E$ and $H$ be 1-forms satisfying the Maxwell's equations on $B_2\backslash\{0\}$, at frequency $\omega>0$ which is not an eigenvalue for the background Maxwell's operator, 
\begin{equation}\left\{\begin{split}
&\curl E=i\omega H,\quad \curl H=-i\omega E\qquad \mbox{on }\;B_2\\
&\nu\times E|_{\partial B_2}= f,
\end{split}\right.
\end{equation}
and let $E_0$ and $H_0$ be solutions to
\begin{equation}\left\{\begin{split}
&\curl E_0=i\omega\mu_0 H_0,\quad \curl H_0=-i\omega \varepsilon_0E_0+\tilde J\qquad\mbox{on }\; B_1\\
&\nu\cdot E_0|_{\partial B_1}=\nu\cdot H_0|_{\partial B_1}=0
\end{split}\right.
\end{equation}
where $\tilde J$ is supported on $B_{r_1}$ for some $r_1<1$. 
Moreover, suppose that $\tl E$ and $\tl H$ belong to $L^1(B_2;\R^3)$ such that
\begin{equation}
(\tl E, \tl H)=\left\{\begin{split}
&({F_1}_*E, {F_1}_*H), &\textrm{ in } B_2\backslash \overline{B_1}\\
&(E_0, H_0), & \textrm{ in }B_1
\end{split}\right.
\end{equation}
where $F_1$ is the singular transformation given by \eqref{eqn:F1-sing}. 
Then we have
\begin{equation}
\tl E_\rho\rightharpoonup\tl E+\alpha[J]\delta_{\partial B_1}\quad \tl H_\rho\rightharpoonup\tl H+\beta[J]\delta_{\partial B_1}\quad\mbox{ as }\;\rho\rightarrow0
\end{equation}
{in the weak topology of $H^{-1}(B_2)$,} 
where $\delta_{\partial B_1}$ is the Dirac's delta distribution supported on $\partial B_1$ and $\alpha[J]$, $\beta[J]$ are smooth and depending on the source term $J$ (see Remark \ref{rmk:delta-strength}).
Moreover, we have (see Remark \ref{rmk:normal-int})
\begin{equation}\label{eqn:interior-bdry}\nu\cdot (\textrm{curl } \tl H)|_{\partial B_1^-}=\nu\cdot (\textrm{curl }\tl E)|_{\partial B_1^-}=0\end{equation}
and 
\begin{equation}\label{eqn: ext-bdry tantential}\nu\times  \tl H|_{\partial B_1^+}=\nu\times \tl E |_{\partial B_1^+}=0.
\end{equation}

\end{theorem}

We note that despite (\ref{eqn: ext-bdry tantential}), 
\begin{equation}\label{eqn: ext-bdry tantential 2}
{\bf h}_t:=\lim_{\rho\to 0}
\nu\times  \tl H_\rho |_{\partial B_1}\quad \hbox{and}\quad
{\bf e}_t:=\lim_{\rho\to 0}
\nu\times \tl E_\rho |_{\partial B_1}.
\end{equation}
may be non-zero. {(See Remark \ref{rmk:tang_cond})}

\section{Motivation: A physical example of scattering from the half-space.}\label{sec:motif}
Before the analysis of the waves in an approximative cloak, let us
 us consider as a motivation on a simple example of scattering of a plane wave from homogeneous half-space that resembles the approximate cloak near the
 cloaking surface.  
To begin with, we decompose
the space $\R^3$ to half spaces $U_+=\{(x,y,z);\ x>0\}$ and  $U_-=\{(x,y,z);\ x<0\}$, and
their interface $\Sigma=\{(x,y,z);\ x=0\}$. We assume that the electromagnetic
parameters in $U_-$ correspond to vacuum, that is, $\e_-=I$ and
$\mu_-=I$. In $U_+$ the electromagnetic
parameters are given by constant matrices,
\ba
\e_+=\textrm{diag }(\e^+_x,\e^+_y,\e^+_z)\quad \mu_+=\textrm{diag }(\mu^+_x,\mu^+_y,\mu^+_z),
\ea 
Below, $\hat e_x=(1,0,0)$,  $\hat e_y=(0,1,0)$, and  $\hat e_z=(0,0,1)$ 
are the unit coordinate vectors. 

To study the case that is close to the 3-dimensional approximate cloak studied
in this paper, we let
\[\e^+_x=\mu^+_x=2\rho^2,\quad \e^+_y=\mu^+_y=2,\quad\e^+_z=\mu^+_z=2.\] 
To illustrate, 
we consider an incident plane wave whose magnetic
field is parallel to $y$-axis. It comes from the domain 
 $U_-$ to the interface $\Sigma$ with non-zero incident angle and 
 scatters and refracts at the interface. More precisely,
 we let magnetic field be $H_\rho(x,z)=h(x,z)\hat e_y$, where 
 \[h(x,z)=h^+(x,z):=h^+ e^{i(k_x^+ x+k_z^+ z)}\hat e_y \quad\textrm{ for }\; x>0\] 
and
\[h(x,z)=h^-(x,z):=(h^{in} e^{i(k_x^- x+k_z^- z)}+h^{sc}  e^{i(-k_x^- x+k_z^- z)})\hat e_y\quad\textrm{ for }\;x<0.\]
Recall that this corresponds to the decomposition of the wave in $U_-$ into a sum of
incident and reflected plane waves. At frequency $\omega>0$, given the incident amplitude $h^{in}\in \C$
and the $z$-component of the incident wave vector $k^+_z\in (0,\omega)$,
we  consider the behavior of the transmitted wave in $U_-$ with different
 values of $\rho>0$. 

The field $H_\rho(x,z)$ satisfies Maxwell's equations in $U_{\pm}$
if
\beq\label{eq: diff equ 1}
& &\frac 1{\mu_y}\left[\frac \p{\p x}\bigg(\frac 1{\e_z} \frac {\p h}{\p x}\bigg)
+\frac \p{\p z}\bigg(\frac 1{\e_x} \frac {\p h}{\p z}\bigg)\right]+\omega^2 h=0.
\eeq
This implies 
\beq\label{eq: dispersion 1}
\frac 1{\mu_y^\pm\e_x^\pm}(k_z^\pm)^2+\frac 1{\mu_y^\pm\e_z^\pm}(k_x^\pm)^2=\omega^2.
\eeq
Correspondingly, the electric field is of the form $E_\rho=E_x\hat e_x+E_z\hat e_z$ where
\ba
E_x=-i\omega^{-1}\frac 1{\e_x} \frac {\p h}{\p z},\quad E_z=i\omega^{-1}\frac 1{\e_z} \frac {\p h}{\p x}.
\ea
Now the transmission conditions on $\Sigma$,
\ba
& &\hat e_x\times H_\rho|_{\Sigma^-}=\hat e_x\times H_\rho|_{\Sigma^+},\\
& &\hat e_x\times E_\rho|_{\Sigma^-}=\hat e_x\times E_\rho|_{\Sigma^+}
\ea
give $k^-_z=k^+_z$ and
\beq\label{eq: B1}
& & h^{in}+h^{sc}=h^+,
\\ \label{eq: B2}
& & k^-_x(h^{in}-h^{sc})  = \frac 12 k^+_x h^+ .
\eeq

Since $k^-_z=k^+_z$, equation (\ref{eq: dispersion 1})
 implies
 \ba
k_x^+=\sqrt{{\mu_y^+\e_z^+}\left(\omega^2-\frac 1{\mu_y^+\e_x^+}(k_z^+)^2\right)}
=\sqrt{4\omega^2-\frac {1}{\rho^2}(k_z^+)^2}=
ik_z^+\rho^{-1}(1+O(\rho^{-1})).
 \ea
Note that when $\rho\to 0$, the expression under the square root becomes
negative. We have choosen above the positive sign for the imaginary
part of the square. Denote below $t(\rho)=\hbox{Im}\, k_x^+$.
Then $t(\rho)\to \infty$ as $\rho\to 0$.

By solving (\ref{eq: B1}) and (\ref{eq: B2}) we obtain 
\ba
 h^+=\frac {4k^-_x}{2k^-_x+k^+_x}h^{in},
\quad h^{sc}=-\left(1-\frac{4k^-_x} {2k^-_x+k^+_x} \right)h^{in}.
\ea
Thus
\ba
& &H_\rho(x,z)=\frac {4k^-_x h^{in}}{2k^-_x+t(\rho)i} e^{-t(\rho) x+ik_z^- z}\hat e_y,\quad 
\hbox{for }x>0,\\
& &H_\rho(x,z)=\left(e^{ik_x^-x+ik_z^- z}
-\left(1-\frac{4k^-_x} {2k^-_x+t(\rho)i} \right) e^{-ik_x^-x+ik_z^- z}\right)h^{in}
\hat e_y,\quad 
\hbox{for }x<0.\\
\ea
Next, let $\chi_{\R_+}(s)$ be the characteristic function of the set $\R_+$. Since
\ba
\lim_{t\to \infty} \chi_{\R_+}(s)\frac 1t e^{-ts}=\delta_0(s)
\ea
in the sense of distributions in $\mathcal D^\prime(\R)$, we see that 
\ba
\lim_{\rho\to 0} H_\rho(x,z)=
\chi_{\R_-}(x)(e^{ik_x^-x+ik_z^- z}
- e^{-ik_x^-x+ik_z^- z})h^{in}
\hat e_y-
4ik^-_x h^{in}\delta_0(x) \hat e_y,
\ea
in the sense of distributions in $\mathcal D^\prime(\R^3)$.
This means that the magnetic field $H_\rho$ 
tends to a generalized function that is a sum of a measurable function
and a delta distribution
 as $\rho\to 0$.  Note that here the delta distribution component is caused
 by the concentration of the waves in a thin layer near the interface in the region $U_-$.
 Note that the  blow up of the fields at the interface, that causes the delta distribution to
 appear, causes the boundary also to reflect the incoming wave perfectly
 with the reflection coefficient $-1$. Similar considerations to the above one,
 with slightly different setting where $\mu$ is constant in the whole space,
 are done in \cite{ZCWK}. Also, the physical interpretation of 
the  blow-up of the fields at the interface as infinite polarization of the material is analyzed in  \cite{ZCWK} in detail. Our aim is to show that  similar phenomenon appears in the 
3-dimensional  approximate cloak as the cloak tends to the ideal one
and analyze the convergence of the electromagnetic fields in the sense of distributions.

\section{Limiting behavior at the interface}\label{sec:proof}

\subsection{Spherical harmonic expansions}
Through out the rest of the paper, we assume that the cloaked medium is uniform, i.e., $\varepsilon_0$ and $\mu_0$ are constants. Then the EM fields $(\tilde E_\rho^-, \tilde H_\rho^-)$ and $(E_\rho^+, H_\rho^+)$ both admit the spherical harmonic expansions, given by 
\begin{equation}\label{eq:SHE-phy-}\begin{split}
\tl E^-_\rho=&~\eps_0^{-1/2}\sum_{n=1}^{\infty}\sum_{m=-n}^{n}\alpha_n^mM_{n,k\omega}^m+\beta_n^m\nabla\times M_{n,k\omega}^m+p_n^mN_{n, k\omega}^m+q_n^m\nabla\times N_{n,
k\omega}^m\\
\tl H^-_\rho=&~\frac{1}{ik\omega}\mu_0^{-1/2}\sum_{n=1}^{\infty}\sum_{m=-n}^{n}k^2\omega^2\beta_n^mM_{n,k\omega}^m+\alpha_n^m\nabla\times M_{n,k\omega}^m\\
&\hspace{5cm}+k^2\omega^2q_n^mN_{n,k\omega}^m+p_n^m\curl N_{n,k\omega}^m,
\end{split}\end{equation}
for $x\in B_1\backslash\overline{B_{r_1}}$, where $k:=(\mu_0\eps_0)^{1/2}$, and
\begin{equation}\label{eq:SHE-vir+}\begin{split}
E^+_\rho=&\sum_{n=1}^\infty\sum_{m=-n}^n \gamma_n^mM_{n,\omega}^m+\eta_n^m\nabla\times M_{n,\omega}^m+c_n^mN_{n,\omega}^m+d_n^m\nabla\times N_{n,\omega}^m,\\
H^+_\rho=&\frac{1}{i\omega}\sum_{n=1}^\infty\sum_{m=-n}^n\omega^2\eta_n^mM_{n,\omega}^m+\gamma_n^m\nabla\times
M_{n,\omega}^m+\omega^2d_n^mN_{n,\omega}^m+c_n^m\nabla\times
N_{n,\omega}^m.
\end{split}\end{equation}
for $y\in
B_2\backslash\overline{B_\rho}$, where the basis vectors are defined in Appendix \ref{appendix:A} by \eqref{eqn:vec-base}.
Notice that the last two terms in the expansion of $\tl E^-_\rho$ with respect to $N_{n,k\omega}^m$ and $\curl N_{n,k\omega}^m$ represent the radiating EM field generated by $\tilde J$. Hence, the coefficients $p_n^m$ and $q_n^m$ are determined. 

In terms of the vector spherical harmonics \eqref{eq:SHV-UV}, we have the expansion of the boundary condition
\begin{equation}\label{eq:SHE-f}
f=\sum_{n=1}^{\infty}\sum_{m=-n}^{n}{\color{black}S_n}(f_{nm}^{(1)}U_n^m+f_{nm}^{(2)}V_n^m).
\end{equation}
where $S_n:=\sqrt{n(n+1)}$.
By \eqref{eq:bdry-MN}, \eqref{eq:bdry-curlMN-tang} and the boundary condition
$\nu\times E^{+}_\rho|_{\partial B_2}=f$, we obtain
\begin{equation}\label{eq:coeff-sys-1}
c_n^mh_n^{(1)}(2\omega)+\gamma_n^mj_n(2\omega)=f_{nm}^{(1)},\quad
d_n^m\mathcal{H}_n(2\omega)+\eta_n^m\mathcal{J}_n(2\omega)=2f_{nm}^{(2)}.
\end{equation}
The transmission conditions in \eqref{eq:MW-phy-2} reads
\begin{equation}\label{eqn:transmission}\hat x\times \tilde E^+_\rho|_{\partial B_1^+}=\rho(\hat y\times E^+_\rho|_{\partial B_\rho^+})=\hat x\times\tilde E^-_\rho|_{\partial B_1^-},
\end{equation}
which gives
\begin{equation}\label{eq:coeff-sys-2}
\begin{split}
\rho c_n^mh_n^{(1)}(\omega\rho)+\rho\gamma_n^mj_n(\omega\rho)=&\eps_0^{-1/2}(\alpha_n^m j_n(k\omega)+p_n^m h_n^{(1)}(k\omega)),\\
d_n^m\mathcal{H}_n(\omega\rho)+\eta_n^m\mathcal{J}_n(\omega\rho)=&\eps_0^{-1/2}(\beta_n^m\mathcal{J}_n(k\omega)+q_n^m\mathcal{H}_n(k\omega)).
\end{split}\end{equation}
Similarly, the transmission condition on the magnetic field implies
\begin{equation}\label{eq:coeff-sys-3}
\begin{split}
k c_n^m\mathcal{H}_n(\omega\rho)+k\gamma_n^m\mathcal{J}_n(\omega\rho)=&\mu_0^{-1/2}(\alpha_n^m\mathcal{J}_n(k\omega)+p_n^m\mathcal{H}_n(k\omega)),\\
\rho
d_n^mh_n^{(1)}(\omega\rho)+\rho\eta_n^mj_n(\omega\rho)=&\mu_0^{-1/2}(k\beta_n^m
j_n(k\omega)+k q_n^mh_n^{(1)}(k\omega)).
\end{split}\end{equation}
Solving \eqref{eq:coeff-sys-2} and \eqref{eq:coeff-sys-3}, we have
\begin{equation}\label{eq:coeff-c-d-alpha-beta}
\begin{split}c_n^m=t_1\gamma_n^m+t_1'p_n^m,\quad&\alpha_n^m=t_2\gamma_n^m+t_2'p_n^m,\\
d_n^m=t_3\eta_n^m+ t_3'q_n^m,\quad&\beta_n^m=t_4\eta_n^m+t_4'q_n^m,\end{split}\end{equation}
where 
\begin{equation}\label{eq:t}\begin{split}
t_1:=&\frac{1}{D_n}\left[\eps_0^{-1/2}k\mathcal{J}_n(\omega\rho)j_n(k\omega)-\mu_0^{-1/2}\rho j_n(\omega\rho)\mathcal{J}_n(k\omega)\right],\\
t_2:=&\frac{1}{D_n}\left[k\rho\mathcal{J}_n(\omega\rho)h_n^{(1)}(\omega\rho)-k\rho j_n(\omega\rho)\mathcal{H}_n(\omega\rho)\right],\\
t_3:=&\frac{1}{D_n'}\left[\mu_0^{-1/2}k \mathcal{J}_n(\omega\rho)j_n(k\omega)-\eps_0^{-1/2}\rho j_n(\omega\rho)\mathcal{J}_n(k\omega)\right],\\
t_4:=&\frac{1}{D_n'}\left[\rho\mathcal{J}_n(\omega\rho)h_n^{(1)}(\omega\rho)-\rho
j_n(\omega\rho)\mathcal{H}_n(\omega\rho)\right];
\end{split}\end{equation}
and
\begin{equation}\label{eq:t'}
\begin{split}
t_1':=&\frac{1}{D_n}\left[h^{(1)}_n(k\omega )\mathcal{J}_n(k\omega )-\mathcal{H}_n(k\omega )j_n(k\omega )\right],\\
t_2':=&\frac{1}{D_n}\left[\eps_0^{-1/2}k h^{(1)}_n(k\omega )\mathcal{H}_n(\omega\rho)-\mu_0^{-1/2}\rho\mathcal{H}_n(k\omega )h^{(1)}_n(\omega\rho)\right],\\
t_3':=&\frac{1}{D_n'}\left[\mathcal{J}_n(k\omega)h^{(1)}_n(k\omega)-\mathcal{H}_n(k\omega)j_n(k\omega)\right],\\
t_4':=&\frac{1}{D_n'}\left[\mu_0^{-1/2}kh^{(1)}_n(k\omega)\mathcal{H}_n(\omega\rho)-\eps_0^{-1/2}\rho\mathcal{H}_n(k\omega)h^{(1)}_n(\omega\rho)\right],
\end{split}
\end{equation}
with
\begin{equation}\begin{split}
D_n=&\mu_0^{-1/2}\rho h_n^{(1)}(\omega\rho)\mathcal{J}_n(k\omega)-\eps_0^{-1/2}k\mathcal{H}_n(\omega\rho)j_n(k\omega),\\
D_n'=&\eps_0^{-1/2}\rho h_n^{(1)}(\omega\rho)\mathcal{J}_n(k\omega)-\mu_0^{-1/2}k\mathcal{H}_n(\omega\rho)j_n(k\omega).
\end{split}\end{equation}
Plugging into \eqref{eq:coeff-sys-1}, we obtain
\begin{equation}\label{eq:coeff-gamma-eta}
\gamma_n^m=\displaystyle\frac{f_{nm}^{(1)}-p_n^m
t_1'h_n^{(1)}(2\omega)}{t_1h_n^{(1)}(2\omega)+j_n(2\omega)},\;\;\;
\eta_n^m=\displaystyle\frac{2f_{nm}^{(2)}- t_3'
q_n^m\mathcal{H}_n(2\omega)}{t_3\mathcal{H}_n(2\omega)+\mathcal{J}_n(2\omega)}.\end{equation}

In this notes, we are interested in the effect of an active source cloaked in $B_1$, hence assuming no incident wave for simplicity, i.e., the case of zero boundary condition $f=0$. Therefore, we have
\begin{equation}\label{eq:coeff-gamma-eta-no-f}
\gamma_n^m=\displaystyle\frac{-p_n^m
t_1'h_n^{(1)}(2\omega)}{t_1h_n^{(1)}(2\omega)+j_n(2\omega)},\;\;\;
\eta_n^m=\displaystyle\frac{- t_3'
q_n^m\mathcal{H}_n(2\omega)}{t_3\mathcal{H}_n(2\omega)+\mathcal{J}_n(2\omega)}.\end{equation}

\bigskip

We will also need the following asymptotic estimates 
derived from \eqref{eq:asympt-jh} {\color{black}for $n\geq0$} 
\begin{equation}\label{eq:asympt-t34-rho}\begin{split}
t_3
=&\frac{i\pi(n+1)}{\Gamma(n+1/2)\Gamma(n+3/2)n}\left(\frac{\omega}{2}\right)^{2n+1}\rho^{2n+1}{\color{black}\left(1+O_{\rho\rightarrow0}(\rho)\right)},\\
t_4
=& \frac{(2n+1)\sqrt\pi}{\Gamma(n+3/2)\mu_0^{-1/2}k\omega nj_n(k\omega)}\left(\frac{\omega}{2}\right)^{n+1}\rho^{n+1}\left(1+O_{\rho\rightarrow0}(\rho)\right),\\
t_3'
= &\frac{2i\sqrt\pi\left[\mathcal{J}_n(k\omega)h^{(1)}_n(k\omega)-\mathcal{H}_n(k\omega)j_n(k\omega)\right]}{\Gamma(n+1/2)\mu_0^{-1/2}knj_n(k\omega)}\left(\frac{\omega}{2}\right)^{n+1}\rho^{n+1}\left(1+O_{\rho\rightarrow0}(\rho)\right)\\
t_4'
= & -\frac{h_n^{(1)}(k\omega)}{j_n(k\omega)}\left(1+O_{\rho\rightarrow0}(\rho)\right),
\end{split}\end{equation}
{\color{black}where  $|O_{\rho\rightarrow0}(\rho)|\leq C(n)\rho$  for some constant $C(n)$ depending on $n$. }

It suffices to show the proof for the electric field only, and that of magnetic field is obtained by symmetry. 

\subsection{Normal components}
Consider
\begin{equation}\label{eq:int-nE-phy}\int_{B_2\backslash\overline{B_{r_1}}}(\hat x\cdot \tilde E_\rho)\phi \,dx=\int_{B_1\backslash\overline{B_{r_1}}}(\hat x\cdot\tilde E_\rho^-)\phi \,dx+\int_{B_2\backslash \overline{B_1}}(\hat x\cdot\tilde E_\rho^+)\phi \,dx.\end{equation}
for $\phi\in H_0^1(B_2)$, which admits the spherical expansion
\begin{equation}\label{eq:SHE-phi}\phi(x)=\sum_{n=0}^\infty\sum_{m=-n}^n\phi_n^m(|x|)Y_n^m(\hat x),\quad x\in B_2,\end{equation}
and necessarily
\begin{equation}\label{eqn:phi-H1}
\sum_{n=0}^\infty\sum_{m=-n}^n\int_0^2|\phi_n^m(r)|^2rdr<\infty,\qquad \sum_{n=0}^\infty\sum_{m=-n}^n\int_0^2\left|\frac{d}{dr}{\phi_n^m}(r)\right|^2rdr<\infty.
\end{equation}

\bigskip

In the physical region $r_1<|x|<1$, by \eqref{eq:bdry-MN} and \eqref{eq:bdry-curlMN-normal}, we have
\begin{equation}\label{eq:normal-phy-E-}\hat x\cdot\tilde E_\rho^-=\eps_0^{-1/2}\sum_{n=1}^\infty\sum_{m=-n}^nS_n^2\frac{1}{|x|}\left[\beta_n^mj_n(k\omega|x|)+q_n^m h_n^{(1)}(k\omega|x|)\right]Y_n^m(\hat x).\end{equation}
Then the first integral of \eqref{eq:int-nE-phy} is given by
\[
\int_{B_1\backslash\overline{B_{r_1}}}(\hat x\cdot\tilde E_\rho^-)\phi \,dx=\int_{r_1}^1\sum_{n=1}^\infty\sum_{m=-n}^n \tilde\psi_{n,\rho}^m(\tilde r)~d\tilde r
\]
where 
\[ 
\tilde\psi_{n,\rho}^m(\tilde r):=S_n^2\eps_0^{-1/2}\left[\beta_{n}^mj_n(k\omega \tilde r)+q_n^m h_n^{(1)}(k\omega \tilde r)\right]\phi_n^m(\tilde r)\tilde r.
\] 
Moreover, we can show
\begin{lemma} Suppose that $\tilde J$ is supported in $B_{r_1}$ for $r_1<1$. Then we have
\begin{equation}\label{eq:lem1}
\lim_{\rho\rightarrow0}\int_{B_1\backslash\overline{B_{r_1}}}(\hat x\cdot\tilde E_\rho^-)\phi \,dx=
\int_{r_1}^1\sum_{n=1}^\infty\sum_{m=-n}^n \lim_{\rho\rightarrow0}\tilde\psi_{n,\rho}^m(\tilde r)~d\tilde r.
\end{equation}
\end{lemma}
\begin{proof}

Given that the current source $\tilde J$ is supported on $B_{r_1}$, the radiation part (only depending on $J$) of $\tilde E_\rho^-$, that is,
\begin{equation}\label{eq:est-psi-q}\sum_{n=1}^\infty\sum_{m=-n}^nS_n^2\eps_0^{-1/2}q_n^mh_n^{(1)}(k\omega \tilde r)\phi_n^m(\tilde r)\tilde r\end{equation}
 is $C^\infty$ for $|x|>r_1$, hence is uniformly convergent for $\tilde r\in[r_1,1]$ and independent of $\rho$. 
Furthermore, we  \HOX{The word ``should" is removed.- Matti}have  
\begin{equation}\label{eq:est-q-1}
\big|S_n^2q_n^mh_n^{(1)}(k\omega \tilde r)\tilde r\big|
\lesssim_M (1+n)^{-M}, \qquad \hbox{for }  \tl r>r_1
\end{equation}
for arbitrary $M>0$, where $A\lesssim_MB$ represents $A\leq C_MB$ for a constant $C_M$ only depending on $M$. Using \eqref{eq:asympt-jh} (or equivalently \eqref{eq:asympt-jh-1}), let $N_1>0$ be such that $N_1>>k\omega$, then for $n>N_1$, $|h_n^{(1)}(k\omega \tilde r)|$ 
is monotonically decreasing for $\tilde r\in [r_1,1]$, implying
\begin{equation}\label{eq:est-q-2}\begin{split}
|q_n^m|&\lesssim_M S_n^{-2}\frac{1}{|h_n^{(1)}(k\omega r_1)|}(1+n)^{-M}.
\end{split}\end{equation}

By \eqref{eq:coeff-c-d-alpha-beta} and \eqref{eq:coeff-gamma-eta-no-f} we have
\[\beta_{n,\rho}^m=\frac{(t_3t_4'-t_4t_3')\mathcal{H}_n(2\omega)+t_4'\mathcal J_n(2\omega)}{t_3\mathcal H_n(2\omega)+\mathcal{J}_n(2\omega)}q_n^m.\]
From \eqref{eq:asympt-t34-rho}, {\color{black}for $n\geq 0$}
\[\color{black}(t_3t_4'-t_4t_3')
= -i\pi\frac{n+1}{n\Gamma(n+1/2)\Gamma(n+3/2)}\frac{h_n^{(1)}(k\omega)}{j_n(k\omega)}\left(\frac{\omega}{2}\right)^{2n+1}\rho^{2n+1}\left(1+O_{\rho\rightarrow0}(\rho)\right),
\]
which implies {\color{black} for $n>0$
\[\beta_{n,\rho}^m= -\frac{h^{(1)}_n(k\omega)}{j_n(k\omega)}q_n^m\left(1+O_{\rho\rightarrow0}(\rho)\right).\]
Therefore,
\begin{equation}\label{eqn:beta-0}
\beta_{n,0}^m=\displaystyle\lim_{\rho\rightarrow0}\beta_{n,\rho}^m=-\frac{h^{(1)}_n(k\omega)}{j_n(k\omega)}q_n^m\end{equation}
and the convergence is uniform.}
As a consequence, for $\rho<<1<<N_1<n$ 
and $\tilde r\in[r_1,1]$, the other term of $\tl \psi_{n,\rho}^m$ can be bounded by
\[\begin{split}\big|S_n^2\eps_0^{-1/2}\beta_{n,\rho}^mj_n(k\omega \tilde r)\phi_n^m(\tilde r)\tilde r\big|=&\left|\frac{\beta_{n,\rho}^mj_n(k\omega \tilde r)}{q_n^mh_n^{(1)}(k\omega\tl r)}\right|\big|S_n^2\eps_0^{-1/2}q_n^mh_n^{(1)}(k\omega\tl r)\phi_n^m(\tilde r)\tilde r\big|\\
\lesssim & \left|\frac{j_n(k\omega \tilde r)h^{(1)}_n(k\omega)}{j_n(k\omega)h_n^{(1)}(k\omega\tl r)}\right|\big|S_n^2\eps_0^{-1/2}q_n^mh_n^{(1)}(k\omega\tl r)\phi_n^m(\tilde r)\tilde r\big|\\
\lesssim & \big|S_n^2\eps_0^{-1/2}q_n^mh_n^{(1)}(k\omega\tl r)\phi_n^m(\tilde r)\tilde r\big|.
\end{split}\]
Therefore, the series converges uniformly (independent of $\rho$) and by dominated convergence theorem the lemma is proved.
\end{proof}

\begin{remark}\label{rmk:normal-int}
The proof actually shows that 
\[\lim_{\rho\rightarrow0}\int_{B_1\backslash\overline{B_{r_1}}}(\hat x\cdot\tilde E_\rho^-)\phi \,dx=\int_{r_1}^1(\hat x\cdot\tilde E_0^-)\phi~dx\]
where 
\[\hat x\cdot\tilde E_0^-(x):=\eps_0^{-1/2}\sum_{n=1}^\infty\sum_{m=-n}^n S_n^2\frac{1}{|x|}\left[\beta_{n,0}^mj_n(k\omega |x|)+q_n^m h_n^{(1)}(k\omega |x|)\right]Y_n^m(\hat x).\]
By the uniform convergence of \eqref{eqn:beta-0} and estimate \eqref{eq:est-q-2}, we have 
\[\hat x\cdot\tl E_0^-|_{\partial B_1^-}=0.\]
Notice that the medium inside $B_1$ is regular and $\tl J$ is supported away from $\partial B_1$, this proves \eqref{eqn:interior-bdry} in Theorem \ref{thm:main}, which is consistent to Weder's definition \cite{We1}.
\end{remark}

\bigskip

Now we focus on the second integral of \eqref{eq:int-nE-phy}. 
First of all,
\begin{equation}\begin{split}
\int_{B_2\backslash\overline{B_1}}(\hat x\cdot \tilde E_\rho^+)\phi~dx
=&\int_{B_2\backslash\overline{B_\rho}}\frac{1}{b}(\hat y\cdot E_\rho^+(y))\phi\left(F_\rho(y)\right)|DF_\rho(y)|~dy\\
=&\int_\rho^2 \sum_{n=1}^\infty\sum_{m=-n}^n \psi_{n,\rho}^m(r)~dr
\end{split}\end{equation}
where
\[\psi_{n,\rho}^m(r):=S_n^2\left[d_{n,\rho}^mh_n^{(1)}(\omega r)+\eta_{n,\rho}^mj_n(\omega r)\right]\phi_n^m(a+br)~\frac{(a+br)^2}{r}\]
and $d_{n,\rho}^m:=d_n^m$ and $\eta_{n,\rho}^m:=\eta_n^m$ to indicate the $\rho$ dependence.
We separate the series into two parts: let $N_2>0$ be such that $\max\{2\omega, k\omega\}<< N_1<N_2$ (as defined in \eqref{eq:asympt-jh}) and consider
\[ 
I_1(\rho):=\int_\rho^2\sum_{n=1}^{N_2}\sum_{m=-n}^n \psi_{n,\rho}^m(r)~dr,\qquad
I_2(\rho):=\int_\rho^2\sum_{n=N_2+1}^\infty\sum_{m=-n}^n \psi_{n,\rho}^m(r)~dr.
\] 

First we have
\begin{lemma}\label{lem:finite-terms}
\begin{equation}\label{eq:lem2}
\lim_{\rho\rightarrow0}I_1(\rho)=\sum_{n=1}^{N_2}\sum_{m=-n}^nS_n^2\frac{\mu_0^{1/2}\left[\mathcal J_n(k\omega)h_n^{(1)}(k\omega)-\mathcal{H}_n(k\omega)j_n(k\omega)\right]}{kn(n+1)j_n(k\omega)}q_n^m\phi_n^m(1).
\end{equation}
\end{lemma}

\begin{proof}
We first consider the limit
\begin{equation}\label{eq:pf-I1-1}\begin{split}&~\lim_{\rho\rightarrow0}\int_\rho^2\sum_{n=1}^{N_2}\sum_{m=-n}^nS_n^2\eta_{n,\rho}^mj_n(\omega r)\phi_n^m(a+br)\frac{(a+br)^2}{r}~dr\\
&~\quad=\sum_{n=1}^{N_2}\sum_{m=-n}^nS_n^2\lim_{\rho\rightarrow0}\eta_{n,\rho}^m\int_\rho^2j_n(\omega r)\phi_n^m(a+br)\frac{(a+br)^2}{r}~dr.\end{split}\end{equation}
From \eqref{eq:coeff-gamma-eta-no-f}, \eqref{eq:coeff-c-d-alpha-beta} and \eqref{eq:asympt-t34-rho}, for $\rho<<1$ and $0\leq n\leq N_2$, 
\begin{equation}\label{eq:asympt-d-eta-rho}
\begin{split}
|\eta_n^m|=&\left|\frac{-t_3'\mathcal{H}_n(2\omega)q_n^m}{t_3\mathcal H_n(2\omega)+\mathcal{J}_n(2\omega)}\right|
\lesssim_{N_2}|q_n^m|\rho^{n+1},\\
|d_n^m|=&\left|\frac{t_3'\mathcal{J}_n(2\omega)q_n^m}{t_3\mathcal H_n(2\omega)+\mathcal{J}_n(2\omega)}\right|
\lesssim_{N_2}|q_n^m|\rho^{n+1}.
\end{split}\end{equation}
Therefore, by Cauchy-Schwartz and \eqref{eqn:phi-H1}
\[\Big|\eta_{n,\rho}^m\int_\rho^2j_n(\omega r)\phi_n^m(a+br)\frac{(a+br)^2}{r}~dr\Big|
\lesssim \rho^{n+1/2}\rightarrow0,\qquad\mbox{as }\;\rho\rightarrow0.\]
This implies that the limit \eqref{eq:pf-I1-1} is 0.\\

To take care of the other terms in $I_1$, we choose $1>>\rho_1>\rho$ and consider
\[\sum_{n=1}^{N_2}\sum_{m=-n}^nS_n^2\lim_{\rho\rightarrow0}d_{n,\rho}^m\left(\int_\rho^{\rho_1}+\int_{\rho_1}^2\right)h_n^{(1)}(\omega r)\phi_n^m(a+br)\frac{(a+br)^2}{r}~dr.\]
By \eqref{eq:asympt-d-eta-rho}, it is easy to see 
\[\sum_{n=1}^{N_2}\sum_{m=-n}^nS_n^2\lim_{\rho\rightarrow0}d_{n,\rho}^m\int_{\rho_1}^2h_n^{(1)}(\omega r)\phi_n^m(a+br)\frac{(a+br)^2}{r}~dr=0.\]
To see that the terms integrated on $(\rho,\rho_1)$ converge to the right hand side of \eqref{eq:lem2}, 
we first apply integration by parts
\[\begin{split}
d_{n,\rho}^m\int_\rho^{\rho_1}h_n^{(1)}(\omega r)\phi_n^m(a+br)\frac{(a+br)^2}{r}~dr=&d_{n,\rho}^m\phi_n^m(1)A_n(\rho)\\
&-d_{n,\rho}^m\int_\rho^{\rho_1}{\phi_n^m}'(a+br)bA_n(r)~dr
\end{split}\]
where 
\[A_n(r):=\int_r^{\rho_1} h_n^{(1)}(\omega s)(a+bs)^2s^{-1}~ds.\]
For $0\leq n\leq N_2$, and $r<\rho_1<<1$
\[\begin{split}
A_n(r)=&~-\frac{i}{2\sqrt\pi}\Gamma(n+1/2)\left(\frac{2}{\omega}\right)^{n+1}\int_r^{\rho_1}(a+bs)^2s^{-(n+2)}\big(1+O(s)\big)~ds\\
&=-\frac{i}{2\sqrt\pi}\Gamma(n+1/2)\left(\frac{2}{\omega}\right)^{n+1}\frac{a^2}{n+1}r^{-(n+1)}\big(1+O_{r\rightarrow0}(r^{-n})\big).
\end{split}\]
Therefore, as $\rho\rightarrow0$, 
\[A_n(\rho)=-\frac{i\Gamma(n+1/2)}{2\sqrt\pi(n+1)}\left(\frac{2}{\omega}\right)^{n+1}\rho^{-(n+1)}\left(1+O_{\rho\rightarrow0}(\rho)\right)\]
and by \eqref{eq:asympt-d-eta-rho}
\[\begin{split}
\left|d_{n,\rho}^m\int_\rho^{\rho_1}{\phi_n^m}'(a+br)bA_n(r)~dr\right|&~\lesssim ~ \left(\int_1^2|{\phi_n^m}'(\tl r)|^2\tl r~d\tl r\right)^{1/2}\left(\int_\rho^{\rho_1}|A_n(r)|^2\frac{b}{a+br}~dr\right)^{1/2}\\
&~\lesssim_\phi~ \rho^{1/2}\rightarrow0\qquad \mbox{ as }\;\rho\rightarrow0.
\end{split}\]
Together we have
\[\lim_{\rho\rightarrow0}I_1(\rho)=\sum_{n=1}^{N_2}\sum_{m=-n}^nS_n^2\{\lim_{\rho\rightarrow0}d_{n,\rho}^mA_n(\rho)\}\phi_n^m(1).\]
From \eqref{eq:asympt-d-eta-rho} and \eqref{eq:asympt-t34-rho}, we further have as $\rho\rightarrow0$, 
\begin{equation}\label{eq:dnm-asympt}d_{n,\rho}^m
=  \frac{2i\sqrt\pi\left[\mathcal{J}_n(k\omega)h^{(1)}_n(k\omega)-\mathcal{H}_n(k\omega)j_n(k\omega)\right]}{\Gamma(n+1/2)\mu_0^{-1/2}knj_n(k\omega)}\left(\frac{\omega}{2}\right)^{n+1}q_n^m\rho^{n+1}\left(1+O_{\rho\rightarrow0}(\rho)\right),
\end{equation}
implying 
 \[\lim_{\rho\rightarrow0}d_{n,\rho}^mA_n(\rho)=\frac{\mu_0^{1/2}\left[\mathcal{J}_n(k\omega)h^{(1)}_n(k\omega)-\mathcal{H}_n(k\omega)j_n(k\omega)\right]}{kn(n+1)j_n(k\omega)}q_n^m. \]
Therefore, \eqref{eq:lem2} is proved.
\end{proof}

Before considering the tail term $I_2(\rho)$, let us define 
\[B_n(r):=\int_r^2h_n^{(1)}(\omega s)(a+bs)^2s^{-1}~ds,\quad r>\rho.\]
for $n>N_2$ (so we can use \eqref{eq:asympt-jh-1}), similar to $A_n(r)$, we have 
\begin{equation}\label{eqn:pf-lem-I2-Bnrho}
B_n(\rho)=
-\frac{i\Gamma(n+1/2)}{2\sqrt\pi(n+1)}\left(\frac{2}{\omega}\right)^{n+1}\rho^{-(n+1)}\left(1+O_{\rho\rightarrow0}(\rho)\right),
\end{equation}
and
\begin{equation}\label{eqn:pf-lem-I2-Bnr}
B_n(r)=-\frac{i}{2\sqrt\pi}\Gamma(n+1/2)\left(\frac{2}{\omega}\right)^{n+1}\frac{a^2}{n+1}r^{-(n+1)}\big(1+O_{r\rightarrow0}(r^{-n})\big)
\end{equation}
for $n>N_2$. 
\begin{lemma}\label{lem:infinite-terms}
\begin{equation}\label{eq:lem-I2}\begin{split}
\lim_{\rho\rightarrow0}I_2(\rho)=&\sum_{n=N_2+1}^\infty\sum_{m=-n}^nS_n^2\frac{\mu_0^{1/2}\left[\mathcal J_n(k\omega)h_n^{(1)}(k\omega)-\mathcal{H}_n(k\omega)j_n(k\omega)\right]}{kn(n+1)j_n(k\omega)}q_n^m\phi_n^m(1).
\end{split}\end{equation}
\end{lemma}
\begin{remark}\label{rmk:delta-strength}
Along with Lemma \ref{lem:finite-terms}, this shows that the limit of the normal component of the exterior field is some function (or distribution) times $\delta(\tilde r-1)$. The smoothness of this function (the strength of the delta singularity) is estimated by the growth of the coefficient with respect to $n>> 1$. For $n\geq N_2$, by \eqref{eq:asympt-jh}, \eqref{eq:asympt-JH} and \eqref{eq:est-q-2}, we have
\[\begin{split}
&\left|S_n^2\frac{\mu_0^{1/2}\left[\mathcal J_n(k\omega)h_n^{(1)}(k\omega)-\mathcal{H}_n(k\omega)j_n(k\omega)\right]}{kn(n+1)j_n(k\omega)}\right|
\\
&=
\frac{\mu_0^{1/2}\frac{(2n+1)\Gamma(n+1/2)}{2\Gamma(n+3/2)}(k\omega)^{-1}}{kn(n+1)\frac{\sqrt\pi}{2\Gamma(n+1/2)}\left(\frac{k\omega}{2}\right)^n}S_n^2|q_n^m|
\left(1+O_{n\rightarrow \infty }(\frac 1n)\right)
\lesssim_M \frac{2n+1}{kn(n+1)}r_1^{n+1}(1+n)^{-M}
\end{split}
\]
for any $M>0$. 
\end{remark}

\begin{proof}[Proof of Lemma \ref{lem:infinite-terms}]
Taking into account the $n$-dependence, we have from \eqref{eq:coeff-gamma-eta-no-f}, \eqref{eq:coeff-c-d-alpha-beta} and \eqref{eq:asympt-t34-rho},
\begin{equation}\label{eq:eta-d-n-rho}
|\eta_{n,\rho}^m|
\lesssim ~\Gamma(n+1/2)\Gamma(n+3/2)\omega^{-n}|q_n^m|\left(\frac{\rho}{k\omega}\right)^{n+1},
\quad
|d_{n,\rho}^m|
\lesssim ~ |q_n^m|\left(\frac{\rho}{k}\right)^{n+1}
\end{equation}
where the general constants associated to $\lesssim$ are independent of $n$.

First to show
\begin{equation}\label{eq:NtInfty-int-1}
I_2(\rho)=\sum_{n=N_2+1}^\infty\sum_{m=-n}^n \Psi_{n,\rho}^m,\qquad\Psi_{n,\rho}^{m}:=\int_{\rho}^{2}\psi_{n,\rho}^{m}(r)~dr,
\end{equation}
by \eqref{eq:eta-d-n-rho} and \eqref{eq:asympt-jh}, we have for $r\in[\rho,2]$ and $n>N_2$
\[\left|d_{n,\rho}^mh_n^{(1)}(\omega r)+\eta_{n,\rho}^mj_n(\omega r)\right|
\lesssim \Gamma(n+1/2)\left(\frac{2}{k\omega}\right)^{n+1}|q_n^m|\rho^{n+1}r^{-(n+1)}.\]
By the estimates \eqref{eq:est-q-2} for $q_{n}^{m}$ and $N_1<N_2$, 
\[|h_n^{(1)}(k\omega r_1)|\gtrsim \Gamma(n+1/2)\left(\frac{2}{k\omega r_1}\right)^{n+1},\]
we have
\[\begin{split}
\int_\rho^2|\psi_{n,\rho}^{m}(r)|~dr &~\lesssim S_{n}^{2}\Gamma(n+1/2)\left(\frac{2}{k\omega}\right)^{n+1}|q_{n}^{m}|\rho^{n+1}\int_\rho^2|\phi_{n}^{m}(a+br)|r^{-n-2}~dr\\
&~\lesssim_M (1+n)^{-M}r_1^{n+1}\rho^{-1/2},
\end{split}\]
implying that the series is uniformly convergent. So \eqref{eq:NtInfty-int-1} is valid.\\

To show 
\begin{equation}\label{eq:NtInfty-int-2}
\lim_{\rho\rightarrow0}\sum_{n=N_2+1}^\infty\sum_{m=-n}^n\Psi_{n,\rho}^m=\sum_{n=N_2+1}^\infty\sum_{m=-n}^n\lim_{\rho\rightarrow0}\Psi_{n,\rho}^m
\end{equation}
we write for $n\geq N_2$
\[\begin{split}\Psi_{n,\rho}^{m}
=&~S_{n}^{2}\eta_{n,\rho}^{m}\int_{\rho}^{2} j_n(\omega r)\phi_n^m(a+br)(a+br)^2r^{-1}~dr\\
&~+S_n^2d_{n,\rho}^m\int_\rho^2 h_n^{(1)}(\omega r)\phi_n^m(a+br)(a+br)^2r^{-1}~dr\\
:=&~ \Psi_{1,n,\rho}^m+\Psi_{2,n,\rho}^m.
\end{split}\]
By \eqref{eq:eta-d-n-rho} and \eqref{eq:asympt-jh}, we have
\[\begin{split}|\Psi_{1,n,\rho}^m|\lesssim &~S_n^2|q_n^m|\Gamma(n+1/2)\left(\frac{\rho}{2k\omega}\right)^{n+1}\int_{\rho}^2r^{n-1}|\phi_n^m(a+br)|(a+br)^2~dr\\
\lesssim &~S_n^2|q_n^m|\frac{\Gamma(n+1/2)}{2n-1}\left(\frac{\rho}{2k\omega}\right)^{n+1}\big(2^n+O(\rho^{n-1/2})\big)\\
\lesssim &~S_n^2|q_n^m|\frac{\Gamma(n+1/2)}{n}\left(\frac{\rho}{k\omega}\right)^{n+1}
\end{split}\]
where the general constants are independent of $n>N_2$ and $\rho$.
The right hand side is summable with respect to $n$ uniformly in $\rho$, using \eqref{eq:est-q-2}. Moreover, it converges to 0 as $\rho\rightarrow0$. Therefore,
\begin{equation}\label{eq:exch-Psi-1}\lim_{\rho\rightarrow0}\sum_{n=N_2+1}^\infty\sum_{m=-n}^n\Psi_{1,n,\rho}^m=\sum_{n=N_2+1}^\infty\sum_{m=-n}^n\lim_{\rho\rightarrow0}\Psi_{1,n,\rho}^m=0.\end{equation}
For the second term, by the integration by parts, we have 
\[\Psi_{2,n,\rho}^m=S_n^2d_{n,\rho}^m\left\{\phi_n^m(1)B_n(\rho)-\int_\rho^2{\phi_n^m}'(a+br)bB_n(r)~dr\right\}.\]
Combining \eqref{eq:eta-d-n-rho} and \eqref{eq:asympt-jh}, 
we obtain 
\[\begin{split}
\left|S_n^2d_{n,\rho}^mB_n(\rho)\phi_n^m(1)\right|&\lesssim S_n^2\left(\frac{2}{k\omega}\right)^{n+1}\frac{\Gamma(n+1/2)}{n+1}|q_n^m\phi_n^m(1)|\\
&\lesssim \left(\frac{2}{k\omega}\right)^{n+1}\frac{\Gamma(n+1/2)}{n+1}\frac{1}{|h_n^{(1)}(k\omega)|}|S_n^2 q_n^m h_n^{(1)}(k\omega)\phi_n^m(1)|\\
&\lesssim~ |S_n^2 q_n^m h_n^{(1)}(k\omega)\phi_n^m(1)|
\end{split}\]
which is summable by that of \eqref{eq:est-psi-q}.
Also, by \eqref{eq:eta-d-n-rho}, \eqref{eqn:pf-lem-I2-Bnr} and \eqref{eq:est-q-1}, one has
\[\begin{split}
&\left|S_n^2d_{n,\rho}^m\int_\rho^2{\phi_n^m}'(a+br)bB_n(r)~dr\right|\lesssim~S_n^2|q_n^m|\left(\frac{\rho}{k}\right)^{n+1}\frac{\Gamma(n+1/2)}{n+1}\left(\frac{2}{\omega}\right)^{n+1}\rho^{-(n+1/2)}\\
&\lesssim~S_n^2|q_n^m|\frac{\Gamma(n+1/2)}{n+1}\left(\frac{2}{k\omega}\right)^{n+1}\rho^{1/2}\lesssim_M~(1+n)^{-(M+1)}\rho^{1/2}
\end{split}\]

By Lebesgue dominated convergence theorem,  
we have 
\[\lim_{\rho\rightarrow0}\sum_{n=N_2+1}^\infty\sum_{m=-n}^n\Psi_{2,n,\rho}^m=\sum_{n=N_2+1}^\infty\sum_{m=-n}^n\lim_{\rho\rightarrow0}\Psi_{2,n,\rho}^m=\sum_{n=N_2+1}^\infty\sum_{m=-n}^nS_n^2B_n^m\phi_n^m(1)\]
where 
\[\begin{split}B_n^m:=&\lim_{\rho\rightarrow0}d_{n,\rho}^mB_n(\rho)\\
=&~\frac{2i\sqrt\pi\left[\mathcal{J}_n(k\omega)h^{(1)}_n(k\omega)-\mathcal{H}_n(k\omega)j_n(k\omega)\right]}{\Gamma(n+1/2)\mu_0^{-1/2}knj_n(k\omega)}\left(\frac{\omega}{2}\right)^{n+1}q_n^m\lim_{\rho\rightarrow0}[\rho^{n+1}B_n(\rho)]
\end{split}\]
by \eqref{eq:dnm-asympt}, which proves the lemma by \eqref{eqn:pf-lem-I2-Bnrho}.
\end{proof}

\begin{remark}\label{rmk:tang_cond}
{Using the above analysis, we can  
calculate the limit of the tangential components of the fields when we are cloaking an active source.
Thus, one can obtain explicit representations for the the limits of right hand sides of equations \eqref{eq:coeff-sys-2} and \eqref{eq:coeff-sys-3}.}
For example, we have
\[T_{n,m}^{(1)}(k,\omega,J):=\lim_{\rho\rightarrow0}\beta_{n,\rho}^m\mathcal{J}_n(k\omega)+q_n^m\mathcal{H}_n(k\omega)=\beta_{n,0}^m\mathcal{J}_n(k\omega)+q_n^m\mathcal{H}_n(k\omega)\neq0\]
where $\beta_{n,0}^m=-\frac{h^{(1)}_n(k\omega)}{j_n(k\omega)}q_n^m$ as in \eqref{eqn:beta-0}. Notice the convergence is uniform in $n$. Therefore, we have
\begin{equation}\lim_{\rho\rightarrow0}\hat x\times\tilde E_\rho^{\pm}|_{\partial B_1}=\eps_0^{-1/2}\sum_{n=1}^\infty\sum_{m=-n}^n S_n(T^{(1)}_{n,m}V_n^m+T^{(2)}_{n,m}U_n^m)\end{equation}
where 
\[T_{n,m}^{(2)}(k,\omega,J):=\lim_{\rho\rightarrow0}\alpha_{n,\rho}^m j_n(k\omega)+p_n^m h^{(1)}_n(k\omega)\]
Similarly, we can obtain the tangential magnetic boundary condition at the interface as well. 
\end{remark}

\appendix
\section{Spherical harmonics and Bessel functions}\label{appendix:A}

Our arguments rely heavily on expanding the EM fields into series of
spherical wave functions. To that end, we introduce for
$n\in\mathbb{Z}^+$ and $m\in\mathbb{Z}$,
\begin{equation}\label{eqn:vec-base}\begin{split}
M_{n,\omega}^m(x):=&~\nabla\times\{xj_n(\omega|x|)Y_n^m(\hat{x})\},\\
N_{n,\omega}^m(x):=&~\nabla\times\{xh_n^{(1)}(\omega|x|)Y_n^m(\hat{x})\},
\end{split}\end{equation}
where $\omega\in\R$ and $\hat{x}=x/|x|$
for $x\in\R^3$. Here, $Y_n^m(\hat{x})$ are spherical harmonics 
and
$h_n^{(1)}(t):=j_n(t)+iy_n(t)$ with $j_n(t)$ and $y_n(t)$, for
$t\in\R$, being the spherical Bessel functions of the first and
second kind, respectively. The following facts about these functions are useful in our estimates. 


Set $S_n=\sqrt{n(n+1)}$. Define 
\begin{equation}
\mathcal{J}_n(t):=j_n(t)+tj_n'(t),\qquad \mathcal{H}_n(t):=h_n^{(1)}(t)+t{h_n^{(1)}}'(t)
\end{equation}
where $j_n'$ and ${h_n^{(1)}}'$ are the derivatives of $j_n$ and ${h_n^{(1)}}$. 
We introduce the vector spherical harmonics
\begin{equation}\label{eq:SHV-UV}
U_n^m(\hat x):=\frac{1}{S_n}\mbox{Grad } Y_n^m(\hat x),\qquad V_n^m(\hat x):=\nu\times U_n^m
\end{equation}
where Grad denotes the surface gradient. They satisfy
\begin{equation}\hat x\times V_n^m(\hat x)=-U_n^m(\hat x),\qquad \hat x\times U_n^m(\hat x)=V_n^m(\hat x).\end{equation}
Moreover, we can rewrite $M_{n,\zeta}^m$ and $N_{n,\zeta}^m$ as 
\begin{equation}\begin{split}
M_{n,\omega}^m(x)=&~-S_n j_n(\omega|x|)V_n^m(\hat x),\\
N_{n,\omega}^m(x)=&~-S_n h_n^{(1)}(\omega|x|) V_n^m(\hat x).
\end{split}\end{equation}
Moreover, we have
\begin{equation}\begin{split}
\curl M_{n,\omega}^m(x)=&S_n|x|^{-1}\mathcal J_n(\omega|x|) U_n^m(\hat x)+S_n^2|x|^{-1}j_n(\omega|x|)Y_n^m(\hat x)\hat x,\\
\curl N_{n,\omega}^m(x)=&S_n|x|^{-1}\mathcal H_n(\omega|x|) U_n^m(\hat x)+S_n^2|x|^{-1} h_n^{(1)}(\omega|x|)Y_n^m(\hat x)\hat x.
\end{split}\end{equation}
Alternatively, we also have by \eqref{eq:SHV-UV}
\begin{equation}\begin{split}
\curl M_{n,\omega}^m(x)=&|x|^{-1}\mathcal{J}_n(\omega|x|)\grad Y_n^m(\hat x)+S_n^2|x|^{-1}j_n(\omega|x|)Y_n^m(\hat x)\hat x,\\
\curl N_{n,\omega}^m(x)=&|x|^{-1}\mathcal H_n(\omega|x|) \grad Y_n^m(\hat x)+S_n^2|x|^{-1} h_n^{(1)}(\omega|x|)Y_n^m(\hat x)\hat x.
\end{split}\end{equation}
It is easy to see that 
\begin{equation}\label{eq:bdry-MN}
\begin{split}
&\hat x\times M_{n,\omega}^m(x)=S_nj_n(\omega|x|)U_n^m(\hat x),\quad \hat x\times N_{n,\omega}^m(x)=S_nh_n^{(1)}(\omega|x|)U_n^m(\hat x)\\
&\hspace{3cm}\hat x\cdot M_{n,\omega}^m( x)=\hat x\cdot N_{n,\omega}^m(x)=0,
\end{split}\end{equation}
and
\begin{equation}\label{eq:bdry-curlMN-tang}
\left\{\begin{array}{l}
\hat x\times (\curl M_{n,\omega}^m(x))=S_n|x|^{-1}\mathcal J_n(\omega|x|) V_n^m(\hat x),\\
\hat x\times (\curl N_{n,\omega}^m(x))=S_n|x|^{-1}\mathcal H_n(\omega|x|) V_n^m(\hat x),
\end{array}\right.
\end{equation}
\begin{equation}\label{eq:bdry-curlMN-normal}
\left\{\begin{array}{l}
\hat x\cdot (\curl M_{n,\omega}^m(x))=S_n^2|x|^{-1}j_n(\omega|x|)Y_n^m(\hat x),\\
\hat x\cdot (\curl N_{n,\omega}^m(x))=S_n^2|x|^{-1}h^{(1)}_n(\omega|x|)Y_n^m(\hat x).
\end{array}\right.
\end{equation}

\bigskip

The spherical Bessel functions are given by 
\[j_n(t)=\sqrt{\pi/(2t)}J_{n+1/2}(t),\qquad y_n(t)=\sqrt{\pi/(2t)}Y_{n+1/2}(t)\]
where $J_{n+1/2}(t)$ and $Y_{n+1/2}(t)$ are the standard Bessel functions. 
More specifically, 
\[j_0(t)=\frac{\sin t}{t}, \qquad  h_0^{(1)}(t)=\frac{\sin t}{t}-i\frac{\cos t}{t}.\]
Let $\Gamma(n+1/2):=\frac{(2n-1)!!}{2^n}\sqrt\pi$. 
From their series representations, we obtain that
\begin{equation}\label{eq:asympt-jh}
j_n(t)\approx \frac{\sqrt{\pi}}{2\Gamma(n+3/2)}\left(\frac{t}{2}\right)^n,\qquad h_n^{(1)}(t)\approx-i\frac{\Gamma(n+1/2)}{2\sqrt\pi}\left(\frac{2}{t}\right)^{n+1}\qquad \mbox{ for }\;n>>t,
\end{equation}
in the following sense
\begin{equation}\label{eq:asympt-jh-1}
\left\{\begin{split}
&j_n(t)=\frac{\sqrt{\pi}}{2\Gamma(n+3/2)}\left(\frac{t}{2}\right)^n\big(1+O(1/n)\big)\\
&\qquad\qquad\mbox{ as  $n\rightarrow\infty$, {\em uniformly for $t$ on a compact subset of $\R$}},\\
&h_n^{(1)}(t)=-i\frac{\Gamma(n+1/2)}{2\sqrt\pi}\left(\frac{2}{t}\right)^{n+1}\big(1+O(1/n)\big)\\
&\qquad\qquad\mbox{ as  $n\rightarrow\infty$, {\em uniformly for $t$ on a compact subset of $(0,\infty)$}}.
\end{split}\right.
\end{equation}
and for each $n>0$, 
\begin{equation}\label{eq:asympt-jh-2}
\left\{\begin{split}
&j_n(t)=\frac{\sqrt{\pi}}{2\Gamma(n+3/2)}\left(\frac{t}{2}\right)^n\big(1+O_{t\rightarrow0}(t)\big)
,\\
&h_n^{(1)}(t)=-i\frac{\Gamma(n+1/2)}{2\sqrt\pi}\left(\frac{2}{t}\right)^{n+1}\big(1+O_{t\rightarrow0}(t)\big)
,
\end{split}\right.
\end{equation}
{\color{black}where $|O_{t\rightarrow0}(t)|\leq C(n) t$ as $t\rightarrow0$ for some constant $C(n)>0$ depending on $n$.} Due to \eqref{eq:asympt-jh-1}, one can easily obtain a uniform $C$ independent of $n$, hence replace $O_n(t)$ by $O(t)$. (Notice that such uniformity for $t$ small is corresponding graphically to the spreading out shape of $j_n(t)$ and $y_n(t)$ with respect to $n$, i.e., less oscillatory for larger $n$.)
In the same sense, we have when $n>>t$,
\begin{equation}\label{eq:asympt-JH}
\mathcal J_n(t)\approx \frac{\sqrt{\pi}(n+1)}{2\Gamma(n+3/2)}\left(\frac{t}{2}\right)^n,\qquad
\mathcal H_n(t)
\approx i\frac{\Gamma(n+1/2)n}{2\sqrt\pi}\left(\frac{2}{t}\right)^{n+1}.
\end{equation}

\medskip

\noindent
{\bf Acknowledgements.}
ML was partly supported by  Academy of Finland, grants 273979 and 284715.
TZ was supported by NSF grant DMS-1501049. 

\end{document}